\documentclass[english]{amsart}

\usepackage{enumerate}
\usepackage{amssymb}
\usepackage{array}
\usepackage{graphicx}
\usepackage{mathrsfs}
\usepackage{multicol}
\usepackage{mathtools}
\usepackage{ mathdots }
\usepackage{color}
\usepackage[dvipsnames,svgnames,table]{xcolor}
\usepackage{comment}
\usepackage[normalem]{ulem}
\usepackage{pgf,tikz,pgfplots}
\usepackage{tikz-cd}
\usetikzlibrary{arrows}
\pgfplotsset{compat=1.15}

\usepackage[curve,all]{xy}
\xyoption{matrix}
 \xyoption{curve}
 \xyoption{color}
 \xyoption{line}
\xyoption{arc}

\usepackage[shortlabels]{enumitem}
\setlist[enumerate]{leftmargin=7mm,topsep=0pt,itemsep=-1ex,partopsep=1ex,parsep=1ex,label=\rm{(\roman*)}}
\setlist[itemize]{leftmargin=5mm,topsep=0pt,itemsep=-1ex,partopsep=1ex,parsep=1ex,label=\raisebox{0.25ex}{\tiny$\bullet$}}

\usepackage[backref=false, colorlinks, linktocpage, citecolor = blue, linkcolor = blue]{hyperref}

\usepackage{marvosym}

\theoremstyle{plain}
\newtheorem{theorem}{Theorem}[section]

\newtheorem*{theoremaux}{Theorem \theoremauxnum}
\gdef\theoremauxnum{1}

\newtheorem*{main-theorem}{Main Theorem}
\newtheorem{proposition}[theorem]{Proposition}
\newtheorem*{propositionaux}{Proposition \propositionauxnum}
\gdef\propositionauxnum{1}

\newtheorem{lemma}[theorem]{Lemma}
\newtheorem*{lemmaaux}{Lemma \lemmaauxnum}
\gdef\lemmaauxnum{1}

\newtheorem{corollary}[theorem]{Corollary}

\newtheorem*{key-problem}{Key Problem}

\theoremstyle{definition}
\newtheorem{definition}[theorem]{Definition}

\newtheorem{example}[theorem]{Example}

\theoremstyle{remark}
\newtheorem{remark}[theorem]{Remark}

\makeatletter
\@namedef{subjclassname@2020}{%
  \textup{2020} Mathematics Subject Classification}
\makeatother

\newcommand{\incl}[1][r]{\ar@<-0.2pc>@{^(-}[#1] \ar@<+0.2pc>@{-}[#1]}

\newcommand{\hs}{\kern 0.8pt}

\renewcommand{\H}{{\mathrm{H}}}

\renewcommand{\P}{\mathbb{P}}

\newcommand{\Hom}{\mathrm{Hom}}

\newcommand{\Spec}{\mathrm{Spec}}

\newcommand{\Sone}{\mathbb{S}^1}

\renewcommand{\k}{\mathrm{k}}

\renewcommand{\H}{\mathrm{H}}

\newcommand{\K}{\mathrm{K}}

\newcommand{\bk}{\overline{\mathrm{k}}}

\newcommand{\A}{\mathbb{A}}

\newcommand{\C}{\mathbb{C}}

\newcommand{\Z}{\mathbb{Z}}
\renewcommand{\K}{\mathrm{K}}
\newcommand{\R}{\mathbb{R}}
\newcommand{\G}{\mathbb{G}}
\newcommand{\Gm}{\mathbb{G}_\mathrm{m}}

\newcommand{\Ga}{\mathbb{G}_\mathrm{a}}

\newcommand{\car}{\mathrm{char}}

\DeclareMathOperator{\SL}{SL}

\DeclareMathOperator{\Aut}{Aut}

\DeclareMathOperator{\Pic}{Pic}

\DeclareMathOperator{\Gal}{Gal}

\DeclareMathOperator{\Ker}{Ker}

\title[Equivariant automorphism group and real forms]{Equivariant automorphism group and real forms\\ of complexity-one varieties}
\date{Version of \today}

\author{Giancarlo Lucchini Arteche}
\address{Departamento de Matem\'aticas, Facultad de Ciencias, Universidad de Chile, Las Palmeras 3425, \~Nu\~noa, Santiago, Chile.}
\email{luco@uchile.cl}

\author{Ronan Terpereau}
\address{Univ. Lille, CNRS, UMR 8524 - Laboratoire Paul Painlev\'e, F-59000 Lille, France}
\email{ronan.terpereau@univ-lille.fr}

\begin{document}

\begin{abstract}
Let $G$ be a connected reductive real algebraic group. We prove that every real $G$-variety of complexity one admits only finitely many pairwise non-isomorphic $(\R,G)$-forms. Our approach relies on representability and structural results for equivariant automorphism groups. More generally, over a perfect field, the equivariant automorphism group of an almost homogeneous variety under a smooth group scheme of finite type is represented by a smooth group scheme of finite type, and is linear whenever the acting group is linear. In characteristic zero, the equivariant automorphism group of every complexity-one variety under a connected reductive group is represented by a smooth group scheme locally of finite type. In the case that is not almost homogeneous, we further describe the subgroup acting trivially on the rational quotient as an extension of an \'etale group scheme, locally isomorphic to $\Z^m$, by a group of multiplicative type.
\end{abstract}



\maketitle

\tableofcontents

\section{Introduction}

Let $X$ be a real algebraic variety, and write
$X_\C=X\times_{\Spec(\R)}\Spec(\C)$ for its complexification. A
\emph{real form}, or $\R$-form, of $X$ is a real algebraic variety $X'$
such that $X'_\C\simeq X_\C$ as complex algebraic varieties.

In this article, we consider this question in the presence of symmetries.
More precisely, let $G$ be a connected reductive $\R$-group and let $X$ be
a real $G$-variety. We study the \emph{$(\R,G)$-forms} of $X$, namely the
real $G$-varieties $X'$ such that $X'_\C\simeq X_\C$ as
$G_\C$-varieties. A natural question is then whether a given $G$-variety
can admit infinitely many pairwise non-isomorphic $(\R,G)$-forms.

The complexity of the $G$-action provides a natural framework for this
question. We recall the relevant notions over an arbitrary field. Let $\k$
be a field, let $G$ be a connected reductive $\k$-group, and let $X$ be a
$G$-variety over $\k$. The \emph{complexity} of $X$ is the codimension of
a general Borel orbit in $X_{\overline{\k}}$. A normal $G$-variety is
called \emph{spherical} if it has complexity zero, or equivalently if a
Borel subgroup of $G_{\overline{\k}}$ has a dense open orbit in
$X_{\overline{\k}}$ (see \cite{Kno91,Per14} for general references).
A homogeneous space $G/H$ is called \emph{spherical} if it is spherical
as a $G$-variety; in this case, $H$ is called a \emph{spherical subgroup}
of $G$.

Returning to the real setting, the finiteness problem behaves very differently depending on the
complexity. Spherical varieties admit only finitely many $(\R,G)$-forms
(see \S~\ref{sec: spherical case}). On the other hand, there exist smooth
algebraic surfaces with infinitely many pairwise non-isomorphic real forms
(see e.g.~\cite{DO19,DOY23,Bot24,TGLOWY}). Taking the acting group to be
trivial yields examples of complexity two with infinitely many
$(\R,G)$-forms, and analogous examples exist in every complexity $c\geq 2$.

This singles out complexity one as the borderline case for the finiteness
problem. Complexity-one $G$-varieties themselves fall into two types: they
are either almost homogeneous, or their general $G$-orbits have
codimension one and are spherical homogeneous spaces. The purpose of this
article is to settle the finiteness problem in both cases.

\smallskip

The finiteness problem is closely related to the structure of the equivariant
automorphism group. Indeed, $(\R,G)$-forms of $X$ are described by the
Galois cohomology of $\Aut^G(X)$ (see \S~\ref{sec: real forms}). This leads
us to study equivariant automorphism groups over a more general base field,
and in particular to the question of their representability.

Let $\k$ be a perfect field and let $X$ be a $G$-variety over $\k$. We
consider the group sheaf $\underline{\Aut}^G(X)$ on the small smooth site
over $\k$, defined by
\[
\underline{\Aut}^G(X)(S)=\Aut_S^{G_S}(X_S),
\]
and ask when it is representable by a smooth $\k$-group scheme.

When $X$ is complete, representability by a $\k$-group scheme locally of
finite type follows from the work of Matsumura and Oort~\cite{MO67}.
However, the varieties considered in this article need not be complete. We
prove representability for every almost homogeneous $G$-variety over a
perfect field and, in characteristic zero, for every complexity-one
$G$-variety.

\smallskip

We now state our main results, beginning with the almost homogeneous case.
In fact, the corresponding representability result holds in greater
generality than the reductive setting considered above.

\begin{theorem}[Theorem~\ref{th: aut of an almost homogeneous variety}]
\label{main th A}
Let $\k$ be a perfect field, let $G$ be a smooth $\k$-group scheme of
finite type, and let $X$ be an almost homogeneous $G$-variety. Then the
group sheaf $\underline{\Aut}^G(X)$ is represented by a smooth
$\k$-group scheme $\Aut^G(X)$ of finite type. Moreover, if $G$ is linear,
then $\Aut^G(X)$ is linear.
\end{theorem}

For spherical $G$-varieties with $G$ connected
and reductive, $\Aut^G(X)$ is more precisely a $\k$-group of multiplicative
type (see \cite[\S5.2]{BP87} and \cite[Th.~7.1]{Kno91}). In complexity one,
explicit computations of $\Aut^{\SL_2}(X)$ for almost homogeneous
$\SL_2$-threefolds can be found in \cite[\S\S\,5.3--5.4]{MJT24}.

As an immediate consequence of Theorem~\ref{main th A}, we obtain the
following finiteness result (see \S~\ref{sec: real forms}).

\begin{corollary}[Corollary~\ref{cor: forms for almost homog varieties}]
Let $\k$ be a field of type $(F)$, let $G$ be a linear algebraic
$\k$-group, and let $X$ be an almost homogeneous $G$-variety. Then $X$
admits only finitely many $(\k,G)$-forms.
\end{corollary}

We now return to the connected reductive setting and consider
complexity-one $G$-varieties that are not almost homogeneous. In this case,
the general $G$-orbits are spherical homogeneous spaces of codimension one,
and the rational quotient is a curve. Besides representability, we obtain a
structural description of the equivariant automorphism group in terms of
this quotient.

\begin{theorem}[Theorem~\ref{thm:full-representability}]\label{main th B}
Let $\k$ be a field of characteristic zero, let $G$ be a connected reductive
$\k$-group, and let $X$ be a complexity-one $G$-variety that is not almost homogeneous. Let $\overline C$ be the unique smooth projective $\k$-curve with
function field
$\k(\overline C)=\k(X)^G$.
Then the following hold.
\begin{enumerate}
    \item
    The group sheaf $\underline{\Aut}^G(X)$ is represented by a smooth
    $\k$-group scheme $\Aut^G(X)$, locally of finite type.

    \item
    The rational quotient $X\dashrightarrow \overline C$ induces a natural
    morphism of $\k$-group schemes
    $    \rho\colon \Aut^G(X)\longrightarrow \Aut(\overline C)$,
    and hence an exact sequence
    \begin{equation*}
        1\longrightarrow \Aut_{\overline C}^G(X)
    \longrightarrow \Aut^G(X)
    \xrightarrow{\rho} \Aut(\overline C),
        \end{equation*}  
    where $\Aut_{\overline C}^G(X)=\Ker(\rho)$ is the subgroup scheme of
    $G$-equivariant automorphisms acting trivially on $\k(X)^G$.

    \item
    There is an exact sequence of abelian $\k$-group schemes
  \begin{equation*}
    1\longrightarrow D
    \longrightarrow \Aut_{\overline C}^G(X)
    \longrightarrow \Lambda
    \longrightarrow 1,
    \end{equation*} 
    where $D$ is of multiplicative type and $\Lambda$ is an \'etale
    $\k$-group scheme that is \'etale-locally isomorphic to the constant
    group scheme $\Z^m$ for some $m\geq 0$.
\end{enumerate}
\end{theorem}

Let us briefly indicate the geometry behind Theorem~\ref{main th B}. We
construct a canonical dense $G$-stable open subset $U\subset X$ admitting
a $(G,H)$-fibration $q\colon U\to C$ over a smooth curve, where $H\subset G$
is spherical. Setting $D=N_G(H)/H$, the study of
$\underline{\Aut}^G(U)$ reduces to that of the equivariant automorphisms of
a $D$-torsor over $C$. Since $D$ is of multiplicative type, we can apply
our general results on equivariant automorphisms of torsors. To pass from
$U$ to $X$, we prove a general extension criterion
(Proposition~\ref{prop:extendable-subgroup}): if a group sheaf of
automorphisms of a schematically dense open subset $U\subset X$ is
represented by a smooth group scheme locally of finite type, then the
subsheaf of automorphisms extending to $X$ is represented by a smooth
closed subgroup scheme.

\begin{remark}
The characteristic-zero assumption in Theorem~\ref{main th B} enters only
through the structure results used to reduce to the torsor setting.
Together with the sphericality of the general orbits, these results provide
a stabilizer in general position $H\subset G$ and a dense open locus on
which the rational quotient is a $(G,H)$-fibration. Sphericality further
implies that $N_G(H)/H\simeq\Aut^G(G/H)$ is of multiplicative type. By
contrast, the representability results for torsors under groups of
multiplicative type, and in particular
Propositions~\ref{prop: Aut^D_Y pour des torseurs}
and~\ref{prop: Aut^T pour des torseurs}, hold over arbitrary perfect fields.
Thus the same strategy applies more generally whenever the required
reduction to the torsor setting is available.
\end{remark}

With these results at hand, we return in \S~\ref{sec: real forms} to the
finiteness problem for real forms. The $(\R,G)$-forms of a real
$G$-variety $X$ are classified by the nonabelian Galois cohomology set
$\H^1(\R,\Aut^G(X))$. We prove in
Proposition~\ref{prop: finiteness of H1R} that $\H^1(\R,\mathcal G)$ is
finite for every $\R$-group scheme $\mathcal G$ locally of finite type such
that $\pi_0(\mathcal G)(\C)$ is finitely generated and virtually abelian.
Combined with the structural description of Theorem~\ref{main th B} and
the almost homogeneous case treated above, this yields our main finiteness
result.

\begin{theorem}[Theorem~\ref{thm: main result on real forms}]
Let $G$ be a connected reductive $\R$-group, and let $X$ be a $G$-variety
of complexity one. Then $X$ admits only finitely many $(\R,G)$-forms.
\end{theorem}

We conclude \S~\ref{sec: real forms} with an example illustrating how the
exact sequences describing $\Aut^G(X)$ can be used to compute explicitly
the $(\R,G)$-forms of a given complexity-one $G$-variety. For a general
strategy in the case of an arbitrary almost homogeneous $G$-variety, we
refer to \cite[\S4.5]{MJT24}.

\subsection*{Notation and setting}\label{sec: notations}
Unless otherwise specified, $\k$ denotes a perfect field.

All group sheaves in this article are considered over the small smooth site over $\k$. 
This means that the sheaves are defined on the category of smooth \(k\)-schemes and satisfy the sheaf condition for smooth coverings.
We follow the Stacks Project definition of smoothness (see \cite[\href{https://stacks.math.columbia.edu/tag/00T2}{Tag 00T2}, \href{https://stacks.math.columbia.edu/tag/01V5}{Tag 01V5}]{stacks-project}), which does not assume that the morphism is of finite type, but only locally of finite type.
The group functors that occur below are nevertheless defined on
all $\k$-schemes, and when we write $\mathcal F(S)$ for a $\k$-scheme $S$
which is not smooth over $\k$, we refer to this underlying group functor.
The sheaf property, unless otherwise specified, always refers to the small
smooth site.

Every group scheme is naturally viewed as a group sheaf via its functor of points. Note that a $\k$-group scheme $G$ and its reduction $G_\mathrm{red}$ (which is also a $\k$-group scheme since the base field $\k$ is perfect) represent the same group sheaf since we work over the small smooth site over $\k$. In particular, every time we mention representability, we assume that the group scheme is smooth.

Given a $\k$-variety $X$ and a smooth $\k$-scheme $S$, we denote by $X_S$ the base change $X\times_\k S$ and by $\underline{\Aut}(X)$ the group sheaf defined by $S\mapsto \Aut_S(X_S)$. This is indeed a sheaf since $X$ can be seen as a sheaf and maps of sheaves glue over any site by \cite[\href{https://stacks.math.columbia.edu/tag/04TQ}{Tag 04TQ}]{stacks-project}.
If moreover, $X$ is endowed with an group scheme action $G\ \rotatebox[origin=c]{-90}{$\circlearrowright$}\ X$, then we denote by $\underline{\Aut}^G(X)$ the subgroup sheaf of $\underline{\Aut}(X)$ defined by $S\mapsto \Aut_{S}^{G_S}(X_S)$.

All cohomology groups (and sets) will be \'etale cohomology groups (or sets).

By a $\k$-curve, we mean a separated, geometrically integral $\k$-scheme of finite type and dimension one.

\subsection*{Use of generative AI}
During the preparation of this article, the authors used OpenAI's GPT-5.6
Sol as a tool for mathematical discussion and editorial assistance. It was
used to examine proofs and hypotheses, explore possible generalizations and
simplifications of arguments, detect minor errors, and improve the
organization and exposition of the manuscript. All mathematical suggestions
were critically assessed and independently verified by the authors, who take
full responsibility for all statements and proofs in the article.

\subsection*{Acknowledgments}
We thank Gary Martínez-N\'u\~nez, Adrien Dubouloz, Diego Izquierdo and Michel Brion for very useful discussions.

The authors gratefully acknowledge the financial support of the ECOS-ANID project no.~C23E07 (France), no.~ECOS230044 (Chile), \emph{“Equivariant Algebraic Geometry with a view towards Birational Geometry and Arithmetics”}. This work was initiated during a visit of the second author to Chile in December 2024, funded by the ECOS-ANID project. It also falls within the general framework of the MathAmSud project no.~24-MATH-08 NPAAG, \emph{“Geometry and Arithmetics of Algebraic Varieties of Non-Positive Curvature”}.

The first author's resarch was partially funded by ANID via Fondecyt Grant N\textsuperscript{o}1240001. The second author acknowledges the support of the CDP C2EMPI, as well as the French State under the France-2030 programme, the University of Lille, the Initiative of Excellence of the University of Lille, the European Metropolis of Lille for their funding and support of the R-CDP-24-004-C2EMPI project.

\section{Representability criteria for group sheaves}
In this section, we establish several representability criteria for group
sheaves that will be used throughout the paper. The main purpose is to show
that naturally occurring subgroups of automorphism group sheaves are
represented by smooth group schemes, locally of finite type. The last result of the section will
be a crucial ingredient in the proofs of Theorems~\ref{main th A} and~\ref{main th B}.

\smallskip

We start with a variant of \cite[Lem.~2.2]{FILA}.

\begin{lemma}\label{lem: des foncteurs en groupes aux corps}
Let $\k$ be a perfect field, and let $\underline{H}$ be a subgroup sheaf of a smooth $\k$-group scheme $G$. Assume that:
\begin{enumerate}
    \item $\underline{H}$ commutes, on $\k$-algebras, with filtered direct limits; and \label{item comm with direct limits}
    \item there exists a smooth $\k$-subgroup scheme $G_0\subset G$ such that for every algebraically closed field $\K\supset\k$, we have $\underline{H}(\K)=G_0(\K)$.\label{item equality on alg closed fields}
\end{enumerate}
Then $\underline{H}$ is represented by the smooth $\k$-group scheme $G_0$.
\end{lemma}

\begin{proof}
We first show that there is a natural inclusion of sheaves $\underline{H} \hookrightarrow G_0$. Since both $\underline{H}$ and $G_0$ are sheaves that commute with filtered direct limits, it suffices to show that for every smooth connected $\k$-scheme $S$, there is an inclusion $\underline{H}(S) \hookrightarrow G_0(S)$.

Let $S$ be a smooth connected $\k$-scheme, and denote by $\K$ an algebraic closure of its function field $\k(S)$. Consider the following commutative diagram:
\[
\xymatrix{
\underline{H}(S) \ar[d] \ar@{^{(}->}[rr] &  & G(S) \ar@{^{(}->}[d] \\
\underline{H}(\K) \ar@{=}[r] & G_0(\K) \ar@{^{(}->}[r] & G(\K)
}
\]
Here, the top horizontal map is injective because $\underline{H}$ is a subsheaf of $G$, and the right vertical map is well-known to be injective. The equality $\underline{H}(\K) = G_0(\K)$ follows from assumption~\ref{item equality on alg closed fields}. It follows that $\underline{H}(S)$ injects into $G(S) \cap G_0(\K)$.

Since $G_0$ is a closed subgroup scheme of $G$ and $S$ is smooth and
connected over $\k$, the scheme $S$ is integral. Hence a morphism
$S\to G$ whose restriction to the generic point factors through $G_0$
factors through $G_0$. Therefore
\[
G(S)\cap G_0(\K)=G_0(S).
\]
Thus, we obtain $\underline{H}(S) \subset G_0(S)$, which gives a natural transformation $\underline{H} \hookrightarrow G_0$ of group sheaves.

To conclude, it remains to show that this inclusion is surjective. If $G_0$ is connected, the result follows from \cite[Lem.~2.2]{FILA}. Assume then that $G_0$ is not connected.

Since $\underline{H}$ commutes with filtered direct limits and $\k$ is perfect, Galois descent reduces the assertion to the case where $\k$ is algebraically closed. In that case, every connected component of $G_0$ contains a $\k$-point.
Let $C$ be such a component and choose
$c\in C(\k)$. By assumption~\ref{item equality on alg closed fields},
$c\in\underline H(\k)$. Translation by $c$ identifies $G_0^\circ$ with
$C$. Since the connected case is already established, we have
\[
(G_0^\circ\cap\underline H)(S)=G_0^\circ(S)
\]
for every smooth $\k$-scheme $S$. Translating by $c$, we obtain
\[
\underline H(S)\cap C(S)=C(S).
\]
Applying this to every connected component of $G_0$ gives
$\underline H(S)=G_0(S)$ for every smooth connected $\k$-scheme $S$.
Finally, since both $\underline H$ and $G_0$ are sheaves, the same equality
holds for every smooth $\k$-scheme $S$.
\end{proof}

We will also make use of the following result, which slightly generalizes \cite[Lem.~2.3]{FILA} and is an application of \cite[Lem.~2.2]{FILA}.

\begin{lemma}\label{lem: lemme FILA representabilite des extensions}
Let $\k$ be a field, and let $G,H$ be smooth $\k$-group schemes. Let $H^\circ$ denote the identity component of $H$ and assume that $H/H^\circ$ is \'etale locally isomorphic to a finitely generated abelian group. Let
\[1\to H\to\mathcal E\to G\to 1,\]
be an exact sequence of group sheaves. Then $\mathcal E$ is representable by a smooth $\k$-group scheme.
\end{lemma}

\begin{proof} 
If $H/H^\circ$ is \'etale locally isomorphic to a torsion-free finitely
generated abelian group, this is exactly \cite[Lem.~2.3]{FILA}. The proof
given there applies more generally as soon as torsors under $H/H^\circ$
are representable.

We verify this under the present hypothesis. By \'etale descent, we may
assume that
$H/H^\circ\simeq \Z^n\times F$,
where $F$ is a finite constant group. Let $P$ be an
$(H/H^\circ)$-torsor. Then $P/F$ is a $\Z^n$-torsor, hence is
representable by the argument of \cite[Lem.~2.3]{FILA}. Moreover,
$P\longrightarrow P/F$
is an $F$-torsor, and is therefore represented by a finite \'etale
morphism. Hence $P$ is representable. The proof of
\cite[Lem.~2.3]{FILA} now applies without further change and shows that
$\mathcal E$ is represented by a smooth $\k$-group scheme.
\end{proof}

We now turn to the representability problem that will be needed in the
proofs of Theorems~\ref{main th A} and~\ref{main th B}: 
given a group of automorphisms of a dense open
subset, we want to show that the subgroup consisting of automorphisms that
extend to the whole variety is representable. We first establish a
constructibility result for the extension locus.

\begin{lemma}\label{lem:bounded-extension-locus}
Let $\k$ be a field, let $X$ be a separated $\k$-scheme of finite
type, let $U\subset X$ be a schematically dense open subset, and let $S$ be a
finite-type $\k$-scheme. Suppose that
$a\colon S\times U\longrightarrow U$
is a family of automorphisms and that the inverse family is algebraic. 
Then there exists a constructible subset $S_{\mathrm{ext}}\subset S$ such
that, for every geometric point $\bar s\to S$, the automorphism
$a_{\bar s}$ extends to an automorphism of $X_{\bar s}$ if and only if the
image of $\bar s$ belongs to $S_{\mathrm{ext}}$. Moreover,
$S_{\mathrm{ext}}$ is compatible with field extension.
\end{lemma}

\begin{proof}
Let
\[
\gamma\colon S\times U\longrightarrow S\times X\times X,
\qquad
(s,x)\longmapsto (s,x,a(s,x)),
\]
be the graph morphism, and let $\Gamma\subset S\times X\times X$ be its
scheme-theoretic image, equivalently its scheme-theoretic closure \cite[\href{https://stacks.math.columbia.edu/tag/01R5}{Tag 01R5}]{stacks-project}. By
\cite[\href{https://stacks.math.columbia.edu/tag/0H41}{Tag 0H41}]{stacks-project},
after a finite locally closed stratification of $S$, formation of $\Gamma$
commutes with arbitrary base change. We may therefore work on one such
stratum, where for every point $s\in S$ the fiber $\Gamma_s$ is the
scheme-theoretic closure of the graph of $a_s$.

Let
$p_1,p_2\colon\Gamma\longrightarrow S\times X$
be the two projections. By \cite[Proposition~9.6.1]{EGAIV3}, the locus of
points $s\in S$ for which
$(p_1)_s\colon\Gamma_s\longrightarrow X_s$
is an isomorphism is locally constructible. This is exactly the locus where
$a_s$ extends to a morphism $X_s\to X_s$. Indeed, if $(p_1)_s$ is an
isomorphism, the extension is $(p_2)_s(p_1)_s^{-1}$. Conversely, if $a_s$
extends to a morphism
$\widetilde a_s\colon X_s\longrightarrow X_s$,
its graph is closed because $X_s$ is separated. Moreover, the graph of
$a_s$ over $U_s$ is schematically dense in it, since schematic density of
$U\subset X$ is preserved by the flat base change
$\k\to\kappa(s)$. Hence this graph is $\Gamma_s$, and $(p_1)_s$ is an
isomorphism.

Applying the same argument to the inverse family and intersecting the two
constructible loci yields precisely the locus $S_{\mathrm{ext}}\subset S$
where $a_s$ extends to an automorphism. Indeed, on this locus the two
extensions are inverse, since their compositions agree with the identity
on the schematically dense open $U_s$.

Finally, the construction is compatible with field extension. On each
stratum, formation of $\Gamma$ commutes with arbitrary base change, and
whether $(p_1)_s$ is an isomorphism is unchanged after a field extension,
since isomorphisms descend under faithfully flat base change. The same
holds for the inverse family. Consequently, for every geometric point
$\bar s\to S$, the automorphism $a_{\bar s}$ extends to
$X_{\bar s}$ if and only if the image of $\bar s$ lies in
$S_{\mathrm{ext}}$, and the formation of $S_{\mathrm{ext}}$ is compatible
with field extension.
\end{proof}

\begin{proposition}\label{prop:extendable-subgroup}
Let $\k$ be a perfect field, let $X$ be a separated $\k$-scheme of finite type,
and let $U\subset X$ be a schematically dense open subset.
Suppose that a sheaf $\mathcal A$ of automorphisms of $U$ is represented by a
smooth $\k$-group scheme $A$, locally of finite type. Then the subsheaf
\[
\mathcal E(S)=
\left\{
g\in A(S)\;\middle|\;
g\text{ extends to an automorphism of }X_S
\right\}
\]
is represented by a smooth closed subgroup scheme $B\subset A$, locally of
finite type.
\end{proposition}

\begin{proof}
Since schematic density is preserved by flat base change and every
$\k$-scheme is flat over $\k$, the open subset $U_S$ is schematically
dense in $X_S$ for every $\k$-scheme $S$.

The assignment $\mathcal E$ is a subgroup sheaf of $A$. Indeed, extensions
are stable under composition and inversion. Moreover, extensions are unique:
two morphisms $X_S\to X_S$ agreeing on the schematically dense open $U_S$
are equal because $X_S$ is separated. Existence of an extension is local
for the smooth topology, by descent for morphisms; uniqueness ensures that
the local extensions agree on overlaps. Thus $\mathcal E$ is a subgroup
sheaf of $A$.

The universal element of $A(A)=\mathcal A(A)$ gives an algebraic family of
automorphisms of $U$. 
Let $\overline\k$ be an algebraic closure of $\k$ and
let $A_i$ be a connected component of $A_{\overline\k}$. Since $A$ is
locally of finite type, $A_i$ is of finite type over $\overline\k$.
Since $U_{\overline\k}\subset X_{\overline\k}$
is schematically dense, Lemma~\ref{lem:bounded-extension-locus}
applies and shows that the subset
\[
E_i=
\left\{
a\in A_i(\overline\k)\;\middle|\;
a|_{U_{\overline\k}}\text{ extends to }X_{\overline\k}
\right\}
\]
is constructible. The union $E=\bigcup_iE_i\subset A(\overline\k)$ is a subgroup.

Let $Z$ be the reduced closed subscheme of $A^\circ_{\overline\k}$ whose
underlying topological space is the Zariski closure of
$E\cap A^\circ(\overline\k)$. Since
$E\cap A^\circ(\overline\k)$ is a subgroup of
$A^\circ(\overline\k)$, its Zariski closure is stable under multiplication
and inversion. Since $\overline\k$ is perfect, products of reduced
$\overline\k$-schemes of finite type are reduced. Hence $Z$ is a closed
subgroup scheme of $A^\circ_{\overline\k}$.

Since $E\cap A^\circ(\overline\k)$ is constructible and Zariski dense in
$Z$, there exists a nonempty open subset $O\subset Z$ such that $O(\overline\k)\subset E\cap A^\circ(\overline\k)$.
Choose $e\in O(\overline\k)$. Then $e^{-1}O$ is an open neighborhood of the
identity in $Z$ such that $(e^{-1}O)(\overline\k)\subset E\cap A^\circ(\overline\k)$.
Since $E\cap A^\circ(\overline\k)$ is a subgroup, its translates by its own
elements show that it is open in $Z(\overline\k)$. An open subgroup is also
closed, since its other cosets are open. By density,
\[
E\cap A^\circ(\overline\k)=Z(\overline\k).
\]

If a connected component $A_i$ meets $E$, choose
$e_i\in E\cap A_i(\overline\k)$. Then
\[
A_i=e_iA^\circ_{\overline\k}\qquad\text{and}\qquad
E\cap A_i(\overline\k)
=
e_i\bigl(E\cap A^\circ(\overline\k)\bigr).
\]
Thus $E\cap A_i(\overline\k)$ is closed in $A_i(\overline\k)$; otherwise it
is empty. Since the connected components of $A_{\overline\k}$ are open,
$E$ is a closed subset of $A(\overline\k)$.

Let $B_{\overline\k}\subset A_{\overline\k}$ be the reduced closed
subscheme with underlying closed subset $E$. Since $E$ is a subgroup of
$A(\overline\k)$, the same argument as above shows that
$B_{\overline\k}$ is a closed subgroup scheme of $A_{\overline\k}$.
The extension condition is invariant under
$\operatorname{Gal}(\overline\k/\k)$, so $B_{\overline\k}$ is Galois
invariant. Its reduced induced structure is therefore equipped with a
descent datum, which is effective for closed subschemes. Thus
$B_{\overline\k}$ descends to a reduced closed $\k$-subgroup scheme
$B\subset A$. Since $\k$ is perfect and $B$ is reduced and locally of finite type, $B$ is
smooth.\\

We conclude by applying Lemma~\ref{lem: des foncteurs en groupes aux corps} to the subgroup
sheaf $\mathcal E\subset A$ and the smooth closed subgroup scheme $B\subset A$. In order to do this, we need to prove that the conditions \ref{item comm with direct limits} and \ref{item equality on alg closed fields} of the lemma are met.

The sheaf $\mathcal E$ commutes, on $\k$-algebras, with filtered direct
limits. Indeed, $X$ is of finite presentation over $\k$, and an extension,
its inverse, and the equalities expressing that they restrict to the given
automorphism on $U$ all descend to a sufficiently large stage of a filtered
direct system. This is condition \ref{item comm with direct limits}.

In order to conclude, let us prove condition \ref{item equality on alg closed fields}, that is, $\mathcal E(\K)=B(\K)$ for every algebraically closed extension $\K/\k$. For each connected component $A_i$ of $A_{\overline\k}$, let
$C_i\subset |A_i|$ be the constructible extension locus given by
Lemma~\ref{lem:bounded-extension-locus}. By construction,
\[
C_i\cap A_i(\overline\k)
=
E\cap A_i(\overline\k)
=
B_{\overline\k}(\overline\k)\cap A_i(\overline\k).
\]
We claim that $C_i=|B_{\overline\k}\cap A_i|$.
Indeed, otherwise the symmetric difference of these two constructible
subsets of the finite-type $\overline\k$-scheme $A_i$ would be nonempty.
It would therefore contain a closed point. Since $\overline\k$ is
algebraically closed, every closed point of $A_i$ is a
$\overline\k$-point, contradicting the preceding equality.

Now let $\Omega/\overline\k$ be an algebraically closed extension. By the
compatibility with field extension in
Lemma~\ref{lem:bounded-extension-locus}, the extension locus in
$(A_i)_\Omega$ is $(C_i)_\Omega = |B_\Omega\cap (A_i)_\Omega|$.
Hence $\mathcal E(\Omega)=B(\Omega)$.

Let now $\K/\k$ be an arbitrary algebraically closed extension and choose an
algebraically closed field $\Omega$ containing both $\overline\k$ and $\K$
over $\k$. For $g\in A(\K)$, we have
\[
g\in\mathcal E(\K)
\quad\Longleftrightarrow\quad
g_\Omega\in\mathcal E(\Omega).
\]
The forward implication is immediate. Conversely, suppose that $g_\Omega$
extends to an automorphism of $X_\Omega$. This extension is unique, so its
two pullbacks to $X_{\Omega\otimes_\K\Omega}$ agree on the schematically dense open
$U_{\Omega\otimes_\K\Omega}$ and hence agree everywhere. Faithfully flat
descent therefore gives an extension of $g$ to an automorphism of $X_\K$.

Since $\mathcal E(\Omega)=B(\Omega)$, we obtain
\[
g\in\mathcal E(\K)
\quad\Longleftrightarrow\quad
g_\Omega\in B(\Omega).
\]
As $B\subset A$ is closed, faithful flatness of $\Omega/\K$ gives
\[
g_\Omega\in B(\Omega)
\quad\Longleftrightarrow\quad
g\in B(\K).
\]
Consequently, $\mathcal E(\K)=B(\K)$ for every algebraically closed extension $\K/\k$. This concludes the proof.
\end{proof}

\section{The case of almost homogeneous varieties}\label{sec: spherical case}
We now prove the representability of the group sheaf $\underline{\Aut}^G(X)$, when $G$ is a smooth $\k$-group scheme of finite type and $X$ is an almost homogeneous $G$-variety over a perfect field $\k$.

Recall that a $G$-variety $X$ is said to be \emph{homogeneous}, respectively \emph{almost homogeneous}, under the $G$-action if the morphism
\[
(\alpha, \operatorname{pr}_2)\colon G \times X \to X \times X, \quad (g,x) \mapsto (g \cdot x, x)
\]
is surjective, respectively dominant, where $\alpha\colon G \times X \to X$ denotes the $G$-action on $X$, and $\operatorname{pr}_2$ is the projection onto the second factor.
In particular, an almost homogeneous $G$-variety contains a unique open
$G$-stable subvariety $X_0$ which is homogeneous under $G$; we call $X_0$
the \emph{open $G$-orbit of $X$}.

We start with the particular case of homogeneous spaces.

\begin{proposition}\label{prop: Aut^G(X) for homogeneous spaces}
Let $\k$ be a perfect field, let $G$ be a smooth $\k$-group scheme of finite type, and let $X$ be a homogeneous $G$-variety. Then the group sheaf $\underline{\Aut}^G(X)$ is representable by a smooth $\k$-group of finite type, denoted $\Aut^G(X)$.
Moreover, if $G$ is linear, then $\Aut^G(X)$ is linear.
\end{proposition}

\begin{proof}
Choose a closed point of $X$. Its residue field is a finite separable
extension of $\k$, since $X$ is of finite type and $\k$ is perfect.
After taking a finite Galois closure, we obtain a finite Galois extension
$L/\k$ such that $X(L)\neq\varnothing$. Choose $x\in X(L)$ and let
$H=(G_L)_x$.
Then $X_L\simeq G_L/H$.

Set
$Q=N_{G_L}(H)/H$.
The normalizer $N_{G_L}(H)$ is a closed $L$-subgroup scheme of $G_L$,
and, since $H$ is normal in $N_{G_L}(H)$, the quotient $Q$ is an
$L$-group scheme of finite type. Right translation defines an action of
$Q$ on $G_L/H$, given by
\[
(gH,nH)\longmapsto gn^{-1}H,
\]
and hence a morphism of group sheaves
\[
\Theta\colon Q\longrightarrow
\underline{\Aut}^{G_L}(X_L).
\]

We claim that $\Theta$ is an isomorphism. Let $S$ be a smooth
$L$-scheme and let $\varphi\in\Aut_S^{G_S}((X_L)_S)$.
The automorphism $\varphi$ is determined by the point
$y=\varphi(x_S)\in (X_L)(S)$,
since $\varphi$ is $G_S$-equivariant.

We show that $y$ defines an $S$-point of $Q$. After an fpqc covering
$T\to S$, lift $y_T$ to an element $g\in G_L(T)$, so that
$y_T=g\cdot x_T$.
Since $\varphi_T$ is $G_T$-equivariant and invertible, the
scheme-theoretic stabilizer of $y_T$ is equal to that of $x_T$, namely
$H_T$. On the other hand, the stabilizer of $g\cdot x_T$ is
$gH_Tg^{-1}$. Hence
$gH_Tg^{-1}=H_T$,
so $g\in N_{G_L}(H)(T)$. Thus $y_T$ comes from $Q(T)$.
Since the morphism
\[
Q=N_{G_L}(H)/H\longrightarrow G_L/H
\]
is a monomorphism, the two pullbacks of the resulting element of $Q(T)$
to $T\times_S T$ coincide. Since $Q$ is an fpqc sheaf, this element
therefore descends uniquely to an element of $Q(S)$.

We therefore obtain a unique element $q\in Q(S)$ whose image in
$(G_L/H)(S)$ is $y=\varphi(x_S)$. By $G_S$-equivariance and the
definition of $\Theta$, we have
$\varphi=\Theta(q^{-1})$.
Thus $\Theta(S)$ is bijective for every smooth $L$-scheme $S$, and
consequently
$\underline{\Aut}^{G_L}(X_L)\simeq Q$
as group sheaves on the small smooth site over $L$.

Replacing $Q$ by $Q_{\mathrm{red}}$ does not change the represented
sheaf on the small smooth site. Since $L$ is perfect,
$Q_{\mathrm{red}}$ is smooth. Moreover, it is of finite type over $L$.
Hence $\underline{\Aut}^{G_L}(X_L)$ is represented by the smooth
finite-type $L$-group scheme $Q_{\mathrm{red}}$.

The conclusion over $\k$ now follows by Galois descent along the finite Galois extension $L/\k$. Since being of finite type descends under finite faithfully flat
base change, $\Aut^G(X)$ is of finite type over $\k$.

Finally, assume that $G$ is linear. Then $N_{G_L}(H)$ is linear, and
the quotient of a linear algebraic group by a closed normal subgroup
is linear. Thus $Q$, and hence $Q_{\mathrm{red}}$, is linear.
Affineness descends under the finite faithfully flat extension $L/\k$,
so $\Aut^G(X)$ is linear.
\end{proof}

\begin{remark}
With the notation from the proposition, let $L/\k$ be a finite Galois extension such that
$X(L)\neq\varnothing$, let $x\in X(L)$, and set $H=(G_L)_x$. Then the
proof gives
\[
\Aut^G(X)_L
\simeq
\bigl(N_{G_L}(H)/H\bigr)_{\mathrm{red}}.
\]
\end{remark}

We now pass from a homogeneous variety to an arbitrary almost homogeneous embedding, using Proposition~\ref{prop:extendable-subgroup}.

\begin{theorem}\label{th: aut of an almost homogeneous variety}
Let $\k$ be a perfect field, let $G$ be a smooth $\k$-group scheme of
finite type, and let $X$ be an almost homogeneous $G$-variety. Then the
group sheaf $\underline{\Aut}^G(X)$ is representable by a smooth
$\k$-group scheme of finite type, denoted $\Aut^G(X)$.
Moreover, if $G$ is linear, then $\Aut^G(X)$ is linear.
\end{theorem}

\begin{proof}
Let $X_0\subset X$ be the open $G$-orbit. By
Proposition~\ref{prop: Aut^G(X) for homogeneous spaces},
$\underline{\Aut}^G(X_0)$ is represented by a smooth $\k$-group scheme
of finite type $\Aut^G(X_0)$.

Every $G_S$-equivariant automorphism of $X_S$ preserves $(X_0)_S$.
Indeed, this can be checked on geometric fibers, where $(X_0)_{\bar s}$
is the open $G_{\bar s}$-orbit.
Hence restriction identifies $\underline{\Aut}^G(X)$ with the subsheaf of
$\underline{\Aut}^G(X_0)$ consisting of the $G$-equivariant automorphisms
of $X_0$ that extend to automorphisms of $X$. By
Proposition~\ref{prop:extendable-subgroup}, this subsheaf is represented
by a smooth closed subgroup scheme of $\Aut^G(X_0)$. Since
$\Aut^G(X_0)$ is of finite type, so is this closed subgroup scheme.

Finally, if $G$ is linear, then
Proposition~\ref{prop: Aut^G(X) for homogeneous spaces} shows that
$\Aut^G(X_0)$ is linear. Since a closed subgroup scheme of a linear
group is linear, $\Aut^G(X)$ is linear as well.
\end{proof}

\section{Torsors under groups of multiplicative type on curves}\label{sec: torseurs}
In this section, we consider $D$-torsors over a perfect base field for $D$ a group of multiplicative type. The main result is the representability of the group of $D$-equivariant automorphisms of a $D$-torsor over a smooth curve (Proposition~\ref{prop: Aut^T pour des torseurs}). 
Besides providing a useful warm-up for what follows, this case will play an essential role in the general argument. Indeed, it will be a crucial ingredient in the proof of the representability of the sheaf $\underline{\Aut}^G(X)$ for arbitrary complexity-one $G$-varieties in the next section.

\smallskip

We start with a simple lemma that actually works for torsors under an arbitrary smooth $\k$-group scheme.

\begin{lemma}\label{lem: lemme relevement morphismes de torseurs}
Let $\k$ be a perfect field, let $S$ be a $\k$-scheme, let $G$ be a smooth $\k$-group scheme, and let $V \to S$ be a $G$-torsor. Denote by $[V]$ the class of $V$ in $\H^1(S,G)$. Let $\varphi \in \Aut(S)$ be an automorphism. Then there exists an automorphism $\psi \in \Aut^G(V)$ making the diagram
\[
\xymatrix@R=1.2em{ 
V \ar[r]^{\psi} \ar[d] & V \ar[d] \\
S \ar[r]^\varphi & S
}
\]
commute if and only if $\varphi^*([V])=[V]$ in $\H^1(S,G)$.
\end{lemma}

\begin{proof}
Assume that such a diagram exists. Since $\varphi^*([V])$ is the class of the torsor obtained by pullback, we have a commutative diagram of $G$-torsors
\[
\xymatrix@R=1.2em{
V \ar@/^1pc/[drr]^{\psi} \ar@/_1pc/[ddr] \ar@{-->}[dr] & & \\
& \varphi^*(V) \ar[r] \ar[d] & V \ar[d] \\
& S \ar[r]^\varphi & S
}
\]
where the dashed arrow is obtained by the universal property of the fiber product. Since every morphism of $G$-torsors is an isomorphism (see \cite[III, Th.~1.4.5.(ii)]{Giraud}), it follows that $\varphi^*([V])=[\varphi^*(V)]=[V]$.

Conversely, assume that $\varphi^*([V])=[V]$. Then there exists a commutative diagram of $G$-torsors
\[
\xymatrix@R=1.2em{
V \ar[r] \ar[d] & \varphi^*(V) \ar[r] \ar[d] & V \ar[d] \\
S \ar@{=}[r] & S \ar[r]^\varphi & S.
}
\]
The composition of the two top horizontal arrows gives the required automorphism $\psi$.
\end{proof}

We first consider automorphisms acting trivially on the base. Let $\k$ be
a perfect field, let $Y$ be a smooth $\k$-scheme, let $D$ be a smooth
$\k$-group of multiplicative type, and let $V\to Y$ be a $D$-torsor. We
define the subgroup sheaf $\underline{\Aut}^D_Y(V)$ by
\[
\underline{\Aut}^D_Y(V)\colon
\{\text{smooth }\k\text{-schemes}\}\to\{\text{Groups}\},
\qquad
S\mapsto\Aut_{Y_S}^{D_S}(V_S).
\]
The following proposition describes this group sheaf explicitly.

\begin{proposition}\label{prop: Aut^D_Y pour des torseurs}
Let $\k$ be a perfect field, let $Y$ be a smooth $\k$-scheme, let $D$ be a
smooth $\k$-group of multiplicative type, and let $V\to Y$ be a $D$-torsor.
Then, for every smooth $\k$-scheme $S$, there is a natural group isomorphism
\[
\underline{\Aut}^D_Y(V)(S)\simeq D(Y_S).
\]
Moreover, if $Y$ is of finite type and geometrically integral, then
$\underline{\Aut}^D_Y(V)$ is represented by a smooth abelian $\k$-group scheme
$\Aut_Y^D(V)$, locally of finite type, which fits into an exact sequence
\[
1\longrightarrow D
\longrightarrow \Aut_Y^D(V)
\longrightarrow \Lambda
\longrightarrow 1,
\]
where $\Lambda$ is an \'etale $\k$-group scheme that is \'etale-locally
isomorphic to the constant group scheme $\Z^m$ for some $m\geq 0$.
\end{proposition}

\begin{proof}
Since $D$ is commutative, the sheaf of $D$-equivariant automorphisms of the
$D$-torsor $V\to Y$ is represented over the fppf site of $Y$ by $D_Y$;
see \cite[III, \S1.5]{Giraud}. Hence, for every smooth $\k$-scheme $S$,
\begin{align*}
\underline{\Aut}^D_Y(V)(S)
&=\Aut^{D_S}_{Y_S}(V_S)\\
&\simeq \Hom_Y(Y_S,D_Y)\\
&\simeq \Hom_\k(Y_S,D)
=D(Y_S).
\end{align*}
This proves the first assertion.

Assume first that $\k$ is algebraically closed. Since $D$ is smooth and of
multiplicative type, there is an isomorphism
$D\simeq T\times F$,
where $T$ is a torus and $F$ is a finite diagonalizable group scheme. 
Let $S$ be a
smooth geometrically connected $\k$-scheme. Since $Y$ and $S$ are
geometrically integral, so is $Y\times_\k S$. Every morphism from
$Y\times_\k S$ to the finite group $F$ is therefore constant. Consequently,
\[
F(Y\times_\k S)=F(\k)=F(S),
\]
and hence
\[
D(Y\times_\k S)/D(S)
\simeq
T(Y\times_\k S)/T(S).
\]
By Rosenlicht's Lemma \cite[App.~A]{FILA}, we have
\[
T(Y\times_\k S)/T(\k)
\simeq
T(Y)/T(\k)\times T(S)/T(\k).
\]
It follows that
\[
D(Y\times_\k S)/D(S)
\simeq
T(Y)/T(\k).
\]
The same description holds componentwise for an arbitrary smooth
$\k$-scheme $S$.

Again by Rosenlicht's Lemma, $T(Y)/T(\k)$ is a finitely generated free
abelian group. Thus there is an exact sequence of group sheaves
\[
1\longrightarrow D
\longrightarrow \underline{\Aut}^D_Y(V)
\longrightarrow \Lambda
\longrightarrow 1,
\]
where $\Lambda$ is the constant sheaf associated with $\Z^m$ for
some $m\geq 0$. Consequently, $\underline{\Aut}^D_Y(V)$ is a $D$-torsor
over $\Lambda$, and is representable by
\cite[III, Th.~4.3.(a)]{MilneEC}; see
\cite[Th.~11.7.1, Rem.~11.8.3]{GrothendieckBrauerIII} for the comparison
between the small smooth and fppf sites. It is smooth and locally of finite
type.

For an arbitrary perfect field $\k$, the preceding exact sequence over a
separable closure $\k_s$ is compatible with the natural Galois action.
Hence it descends to an exact sequence
\[
1\longrightarrow D
\longrightarrow \Aut_Y^D(V)
\longrightarrow \Lambda
\longrightarrow 1,
\]
where $\Lambda$ is an \'etale $\k$-group scheme which becomes a constant
group scheme $\Z^m$ after an \'etale extension of $\k$. This proves
the result.
\end{proof}

\begin{remark}
The first part of Proposition~\ref{prop: Aut^D_Y pour des torseurs} does
not generalize easily to smooth $\k$-group schemes that are not of multiplicative
type, for essentially two reasons. First, the sheaf of $G$-equivariant
$Y$-automorphisms, viewed as a sheaf over $Y$, is represented by a group
scheme that, in general, is not obtained from a $\k$-group scheme by base
change to $Y$, but is instead an inner form of $G_Y$
(see~\cite[III, \S1.5]{Giraud}). Second, and more importantly, the group of
$Y$-points of an arbitrary smooth group scheme can be quite large. For
instance, if $G$ contains a copy of $\Ga$, the functor
$S\longmapsto \Ga(S_Y)$
is not representable by a $\k$-scheme in general.
\end{remark}

We now consider the full group sheaf $\underline{\Aut}^D(V)$. We prove its representability when the base is a smooth curve. We will use the following simple observation.

\begin{lemma}\label{lem: action picard}
Let $\k$ be a perfect field, and let $C$ be a smooth projective $\k$-curve. Then $\underline{\Aut}(C)$ acts naturally on $\Pic_{C/\k}$ by
\[
\varphi\cdot[L]=[(\varphi^{-1})^*L].
\]
In particular, the stabilizer in $\Aut(C)$ of every class in $\Pic(C)$ is closed.
\end{lemma}

\begin{proof}
The action is induced by pullback of line bundles. Since $\Pic_{C/\k}$ is separated, the stabilizer of a point is closed.
\end{proof}

\begin{proposition}\label{prop: Aut^T pour des torseurs}
Let $\k$ be a perfect field, let $C$ be a smooth $\k$-curve, let $D$ be a smooth $\k$-group of multiplicative type, and let $V\to C$ be a $D$-torsor. Then the group sheaf $\underline{\Aut}^D(V)$ is representable by a smooth $\k$-group scheme, locally of finite type. Moreover, the quotient morphism $V\to C=V/D$ induces a morphism of $\k$-group schemes
\[
\rho\colon \Aut^D(V)\longrightarrow \Aut(C),
\]
and there is an exact sequence of $\k$-group schemes
\[
1\to \Aut^D_C(V)\to \Aut^D(V)\xrightarrow{\rho}\Aut(C).
\]
\end{proposition}

\begin{proof}
Since $V$ is a $D$-torsor, we have $C=V/D$. Consequently, for every smooth $\k$-scheme $S$, any $D_S$-equivariant automorphism
$\varphi\colon V_S\longrightarrow V_S$
descends uniquely to an automorphism
$\overline{\varphi}\colon C_S\longrightarrow C_S$
characterized by the commutative diagram
\[
\begin{tikzcd}[row sep=small]
V_S \arrow[r,"\varphi"] \arrow[d] & V_S \arrow[d] \\
C_S \arrow[r,"\overline{\varphi}"] & C_S.
\end{tikzcd}
\]
The assignment $\varphi\mapsto\overline{\varphi}$ is functorial in $S$ and defines a natural morphism of group sheaves
\[
\rho\colon \underline{\Aut}^D(V)\longrightarrow \Aut(C).
\]
Its kernel is precisely $\Aut^D_C(V)$, which is representable by Proposition~\ref{prop: Aut^D_Y pour des torseurs}. We therefore have an exact sequence of group sheaves
\[
1\to \Aut^D_C(V)\to \underline{\Aut}^D(V)\xrightarrow{\rho}\Aut(C).
\]

We show that the image of $\rho$ is a closed subgroup of $\Aut(C)$. By Lemma~\ref{lem: des foncteurs en groupes aux corps}, it is enough to show that the image sheaf commutes with filtered direct limits and that, for every algebraically closed field $\K\supset\k$, its $\K$-points form a closed subgroup of $\Aut(C)(\K)$.

Commutativity with filtered direct limits follows from the fact that $V$ and $D$ are of finite type over $\k$. 
It remains to prove that, after any extension of the base field to an
algebraically closed field, the image of
\[
\rho(\k)\colon \underline{\Aut}^D(V)(\k)\longrightarrow \Aut(C)(\k)
\]
is closed. We may therefore assume that $\k$ is algebraically closed.
Once this closedness is established, we get that the image of
\[
\rho\colon \underline{\Aut}^D(V)\longrightarrow\Aut(C)
\]
is representable by a smooth closed $\k$-subgroup scheme of $\Aut(C)$, hence by a smooth $\k$-group scheme of finite type. Lemma~\ref{lem: lemme FILA representabilite des extensions} then shows that $\underline{\Aut}^D(V)$ is representable by a smooth $\k$-group scheme. Since its kernel $\Ker(\rho)=\Aut^D_C(V)$ is locally of finite type by Proposition~\ref{prop: Aut^D_Y pour des torseurs} and its image in $\Aut(C)$ is of finite type, $\Aut^D(V)$ is locally of finite type, and there is an exact sequence of $\k$-group schemes
\[
1\to \Aut^D_C(V)\to \Aut^D(V)\xrightarrow{\rho}\Aut(C).
\]

We are then only left with proving the closedness of the image of $\rho(\k)$. By Lemma~\ref{lem: lemme relevement morphismes de torseurs}, this image is the stabilizer
\[
\Sigma_V=\bigl\{\varphi\in\Aut(C)(\k)\mid \varphi^*([V])=[V]
\text{ in }\H^1(C,D)\bigr\}.
\]
Since $\k$ is algebraically closed and $D$ is smooth of multiplicative type, we have $D\simeq T\times F$, where $T$ is a torus and $F$ is a finite constant abelian group of order prime to $\car(\k)$. Since the map
\[
H^1(C,D)\to H^1(C,T)\times H^1(C,F)
\]
is an isomorphism, we see that $\Sigma_V$ is the intersection of the subgroups $\Sigma_{V_T}$ and $\Sigma_{V_F}$ of $\Aut(C)(\k)$ stabilizing the classes of
\[
V_T=V/F=V\times^D T
\qquad\text{and}\qquad
V_F=V/T=V\times^D F,
\]
respectively. It then suffices to prove that both are closed.

\smallskip

Let us first treat $V_T$. Since $T\simeq\Gm^n$, the class $[V_T]\in H^1(C,T)$ corresponds to $n$ classes in $\Pic(C)$, and $\Sigma_{V_T}$ is the intersection of their stabilizers. Let $\overline C$ be the smooth projective completion of $C$. If $C=\overline C$, Lemma~\ref{lem: action picard} shows that each stabilizer is closed in $\Aut(C)$. Assume that $C\neq\overline C$. Every automorphism of $C$ extends uniquely to an automorphism of $\overline C$ preserving the finite set $\overline C\setminus C$. If $\Aut(C)$ is infinite, then $\overline C\simeq\P^1$ and $\overline C\setminus C$ has at most two points. Hence $C\simeq\A^1$ or $C\simeq\Gm$, and in both cases $\Pic(C)=0$. Otherwise, $\Aut(C)$ is finite. Thus, in every case, $\Sigma_{V_T}$ is closed in $\Aut(C)(\k)$.

\smallskip

We deal now with $V_F$. The morphism $V_F\to C$ is a finite \'etale $F$-torsor, so $V_F$ is a smooth (possibly disconnected) $\k$-curve with finitely many connected components $W_1,\dots,W_r$, and $C=V_F/F$. By
Lemma~\ref{lem: lemme relevement morphismes de torseurs}, applied to this torsor, $\Sigma_{V_F}$ is the image of
\[\rho_F(\k)\colon \underline{\Aut}^F(V_F)(\k)\longrightarrow\Aut(C)(\k),\]
where $\rho_F:\underline{\Aut}^F(V_F)\to\Aut(C)$ is the morphism of group sheaves constructed above for any group of multiplicative type. Now, $\Aut(V_F)$ is an algebraic group. Indeed, it permutes the components $W_i$, so its subgroup preserving each $W_i$ has finite index
and equals $\prod_{i}\Aut(W_i)$, and each $W_i$ is a $\k$-curve. Since $\Aut^F(V_F)$ is simply the centralizer of $F\subseteq\Aut(V_F)$, we see that it is an algebraic group as well. Thus, $\rho_F(\k)$ comes from a morphism of algebraic groups since both source and target in $\rho_F$ are algebraic groups. Consequently, its image $\Sigma_{V_F}$ is closed.
\end{proof}

\section{The case of non almost homogeneous complexity-one varieties}
Let $\k$ be a field of characteristic zero, unless stated otherwise, let $G$ be a connected reductive $\k$-group, and let $X$ be a $G$-variety of complexity one that is not almost homogeneous.

The aim of this section is to prove that the group sheaf $\underline{\Aut}^G(X)$ is representable; see Theorem~\ref{thm:full-representability} for a precise statement.
The main idea is to isolate an intrinsic dense open subset
of $X$ which is fibered over a smooth curve with fibers isomorphic to a
spherical homogeneous space $G/H$. We first construct this fibration and show
that it is preserved by equivariant automorphisms. We then reduce the
representability of its equivariant automorphism sheaf to the case of torsors
under groups of multiplicative type treated in Proposition~\ref{prop: Aut^T pour des torseurs}. 
Finally, we extend the
result from this dense open subset to $X$ using
Proposition~\ref{prop:extendable-subgroup}.

\begin{remark}\label{rk: car zero}
The characteristic-zero assumption will only enter through the structure
results used below for varieties with spherical general orbits, namely the
existence of a stabilizer in general position and the existence of a
dense open locus over which the rational quotient is a $(G,H)$-fibration (in the proof of Proposition~\ref{prop:useful-curve}).
The remaining representability arguments work more generally over perfect
fields.
\end{remark}

\begin{definition}\label{def:GH-fibration}
Let $\k$ be an arbitrary field, let $G$ be a $\k$-group, and let $H\subset G$ be a
closed $\k$-subgroup. A $(G,H)$-fibration over a $\k$-scheme $C$ is a
$G$-invariant morphism
$q\colon U\to C$
for which there exists an \'etale surjective morphism $C'\to C$ and a
$G$-equivariant isomorphism over $C'$
\[
U\times_C C'\simeq (G/H)\times_{\k} C'.
\]
\end{definition}

\begin{proposition}\label{prop:useful-curve}
Let $\k$ be an algebraically closed field of characteristic zero, let $G$ be
a connected reductive $\k$-group, and let $X$ be a $G$-variety of
complexity one that is not almost homogeneous. Then there exist a
spherical subgroup $H\subset G$, a dense smooth $G$-stable open subset
$U\subset X$, a smooth $\k$-curve $C$, and a $(G,H)$-fibration
$q\colon U\to C$
with the following properties:
\begin{enumerate}
\item the construction of $U$ and $C$ is intrinsic and commutes with
extensions of algebraically closed fields;
\item for every smooth $\k$-scheme $S$, every
$\varphi\in\Aut_S^{G_S}(X_S)$ preserves $U_S$ and induces an automorphism
of $C_S$.
\end{enumerate}
Consequently, there is an exact sequence of group sheaves
\[
1\longrightarrow\underline{\Aut}_C^G(X)
\longrightarrow\underline{\Aut}^G(X)
\longrightarrow\underline{\Aut}(C),
\]
where
\[
\underline{\Aut}_C^G(X)(S)
=
\left\{
\varphi\in\Aut_S^{G_S}(X_S)
\;\middle|\;
q_S\circ\varphi=q_S
\right\},
\]
with $q_S\colon X_S\dashrightarrow C_S$ viewed as a rational map.
\end{proposition}

\begin{proof}
By Rosenlicht's theorem (see \cite{Ros63} or \cite[\S2.3]{PV89}), there is
a dense $G$-stable open subset of $X$ admitting a geometric quotient.
The union of all such open subsets is again a $G$-stable open subset
$V_{\mathrm{Ros}}\subset X$ admitting a geometric quotient
$\theta_{\mathrm{Ros}}\colon V_{\mathrm{Ros}}\to Y_{\mathrm{Ros}}$,
since geometric quotients agree uniquely on overlaps and therefore glue.
Thus $V_{\mathrm{Ros}}$ is the unique maximal Rosenlicht locus. 
Put
$V=V_{\mathrm{Ros}}\cap X_{\mathrm{sm}}$.
Then $V$ is a dense smooth integral $G$-stable open subset of $X$. Let
$\theta\colon V\to Y$ 
be the induced geometric quotient. 
Then $Y$ is an integral $\k$-scheme of finite type and dimension one. Since $Y$ is normal, it is regular, and hence smooth over $\k$, as $\k$ is perfect.

By \cite[Theorem~3.1]{AB05} (see also \cite[Theorem~3.1]{Lan20}), there is
a spherical subgroup $H\subset G$, unique up to conjugacy, which is the
stabilizer in general position. Let $Y_H\subset Y$ be the union of the open
subsets $W\subset Y$ for which
$\theta^{-1}(W)\to W$
is a $(G,H)$-fibration. By \cite[Theorem~2.13]{CKPR11}, this locus is nonempty, thus dense. 
Since being a $(G,H)$-fibration is \'etale-local on the base and
stable under base change, $Y_H$ is the unique maximal open subset of $Y$
with this property.

Let $\overline C$ be the unique smooth projective $\k$-curve with function
field $\k(Y)$. The identification $\k(\overline C)=\k(Y)$ induces a
birational morphism
$\psi\colon Y\to \overline C$.
Hence the set
\[
D=
\left\{
c\in\overline C
\;\middle|\;
\#\bigl(Y_H\cap\psi^{-1}(c)\bigr)>1
\right\}
\]
is finite. Put
\[
C=Y_H\setminus\psi^{-1}(D),
\qquad
U=\theta^{-1}(C),
\qquad
q=\theta|_U.
\]
The restriction $\psi|_C$ is an injective local isomorphism, hence an open
immersion. Thus $C$ is a smooth $\k$-curve, and
$q\colon U\to C$
is a $(G,H)$-fibration.

All these constructions are intrinsic: $V_{\mathrm{Ros}}$ is the maximal
Rosenlicht locus, $V=V_{\mathrm{Ros}}\cap X_{\mathrm{sm}}$, the conjugacy
class of $H$ is the stabilizer in general position, $Y_H$ is the maximal
$(G,H)$-fibration locus, $\overline C$ is the smooth projective model of
$\k(Y)$, and $D$ is characterized by the local isomorphism
$Y_H\to\overline C$. 
Moreover, the formation of all these objects is compatible with
automorphisms of the algebraically closed base field and commutes with
extensions of algebraically closed fields. In particular, for every
algebraically closed extension $\K/\k$, applying the above construction to
$X_\K$ yields exactly the base changes of the corresponding objects for
$X$. In particular, the resulting open subset and curve are
$U_\K$ and $C_\K$, respectively.

Let now $S$ be a smooth $\k$-scheme and let
$\varphi\in\Aut_S^{G_S}(X_S)$.
For every geometric point $s\to S$, the induced automorphism
$\varphi_s\in\Aut^{G_s}(X_s)$
preserves $U_s$, by the intrinsic characterization above and its compatibility
with extensions of algebraically closed fields. Hence $\varphi(U_S)$ and
$U_S$ have the same underlying open subset of $X_S$, and therefore
$\varphi(U_S)=U_S$
as open subschemes.

Since $q_S\colon U_S\to C_S$ is \'etale-locally on $C_S$ the projection
\[
(G_S/H_S)\times_S C'\to C',
\]
every $G_S$-invariant morphism from $U_S$ factors uniquely after the
\'etale base change $C'\to C_S$. The resulting local factorizations agree
on $C'\times_{C_S}C'$ by uniqueness, and therefore descend to $C_S$.
Hence $q_S$ is a categorical quotient by $G_S$. Consequently the $G_S$-invariant morphism
$q_S\circ\varphi\colon U_S\to C_S$
factors uniquely through $q_S$, giving a morphism
$\overline\varphi\colon C_S\to C_S$
such that
$q_S\circ\varphi=\overline\varphi\circ q_S$.
Applying the same argument to $\varphi^{-1}$ shows that
$\overline\varphi$ is an automorphism. Thus we obtain a morphism of group
sheaves
\[
\underline{\Aut}^G(X)\longrightarrow\underline{\Aut}(C).
\]
Its kernel consists precisely of the equivariant automorphisms satisfying
$q_S\circ\varphi=q_S$, namely $\underline{\Aut}_C^G(X)$. This proves the
exact sequence.
\end{proof}

We now study the equivariant automorphisms of the $(G,H)$-fibration obtained
above. The sphericality of $G/H$ is crucial here, as it implies that $N_G(H)/H$ is of multiplicative type,
reducing the problem to the torsor case considered in the previous section
(Propositions~\ref{prop: Aut^D_Y pour des torseurs} and~\ref{prop: Aut^T pour des torseurs}).

\begin{lemma}\label{lem:friendly-aut}
Let $\k$ be a field of characteristic zero and let
$q\colon U\to C$ be a $(G,H)$-fibration over a smooth $\k$-curve, where
$G$ is connected reductive and $G/H$ is spherical. Then
$\underline{\Aut}^G(U)$ is represented by a smooth group scheme, locally of finite type.
\end{lemma}

\begin{proof}
Put
\[
D=N_G(H)/H
\qquad\text{and}\qquad
V=U^H.
\]
By \cite[Propositions~2.9--2.10]{CKPR11}, the morphism $V\to C$ is a
$D$-torsor and there is a $G$-equivariant isomorphism
\[
U\simeq (G/H)\times^D V.
\]
These constructions commute with base change. Moreover, a $G_S$-equivariant
automorphism of $U_S$ preserves the subscheme $V_S=U_S^H$, and hence induces
a $D_S$-equivariant automorphism of $V_S$. Conversely, every $D_S$-equivariant
automorphism of $V_S$ induces a $G_S$-equivariant automorphism of
$(G_S/H_S)\times^{D_S}V_S\simeq U_S$. These constructions are inverse to
each other, and therefore
\[
\underline{\Aut}^G(U)\simeq\underline{\Aut}^D(V).
\]

Since $H$ is spherical,
\cite[Corollary following Proposition~5.2]{BP87} implies that $D$ is a group of
multiplicative type. The assertion then follows from
Proposition~\ref{prop: Aut^T pour des torseurs}.
\end{proof}

Combining the previous results with the extension result
Proposition~\ref{prop:extendable-subgroup} now yields the main result of this
section.

\begin{theorem}\label{thm:full-representability}
Let $\k$ be a field of characteristic zero, let $G$ be a connected reductive
$\k$-group, and let $X$ be a complexity-one $G$-variety that is not almost homogeneous. Let $\overline C$ be the unique smooth projective $\k$-curve with
function field
$\k(\overline C)=\k(X)^G$.
Then the following hold.
\begin{enumerate}
    \item\label{main th item i}
    The group sheaf $\underline{\Aut}^G(X)$ is represented by a smooth
    $\k$-group scheme $\Aut^G(X)$, locally of finite type.

    \item\label{main th item ii}
    The rational quotient $X\dashrightarrow \overline C$ induces a natural
    morphism of $\k$-group schemes
    $    \rho\colon \Aut^G(X)\longrightarrow \Aut(\overline C)$,
    and hence an exact sequence
    \begin{equation*}
        1\longrightarrow \Aut_{\overline C}^G(X)
    \longrightarrow \Aut^G(X)
    \xrightarrow{\rho} \Aut(\overline C),
        \end{equation*}  
    where $\Aut_{\overline C}^G(X)=\Ker(\rho)$ is the subgroup scheme of
    $G$-equivariant automorphisms acting trivially on $\k(X)^G$.

    \item\label{main th item iii}
    There is an exact sequence of abelian $\k$-group schemes
  \begin{equation*}
    1\longrightarrow D
    \longrightarrow \Aut_{\overline C}^G(X)
    \longrightarrow \Lambda
    \longrightarrow 1,
    \end{equation*} 
    where $D$ is of multiplicative type and $\Lambda$ is an \'etale
    $\k$-group scheme that is \'etale-locally isomorphic to the constant
    group scheme $\Z^m$ for some $m\geq 0$.
\end{enumerate}
\end{theorem}

\begin{proof}
We may assume that $\k$ is algebraically closed, the general case following
by Galois descent at the end of the proof.

By Proposition~\ref{prop:useful-curve}, there exists a nonempty $G$-stable
open subset $U\subset X$, a nonempty open subscheme
$C\subset \overline C$, and a $(G,H)$-fibration
$q\colon U\longrightarrow C$
for some spherical subgroup $H\subset G$. In particular,
$\k(C)=\k(\overline C)=\k(X)^G$.
Put
\[
\Delta=N_G(H)/H
\qquad\text{and}\qquad
V=U^H.
\]
As in the proof of Lemma~\ref{lem:friendly-aut}, the morphism $V\to C$ is a
$\Delta$-torsor, $\Delta$ is a smooth group of multiplicative type, and
restriction to $V$ yields an isomorphism of group sheaves
$\underline{\Aut}^G(U)
\simeq
\underline{\Aut}^{\Delta}(V)$.
By Proposition~\ref{prop: Aut^T pour des torseurs}, this sheaf is represented by
a smooth $\k$-group scheme
$\Aut^G(U)\simeq \Aut^\Delta(V)$,
locally of finite type.

By Proposition~\ref{prop:extendable-subgroup}, restriction from $X$ to $U$
identifies $\underline{\Aut}^G(X)$ with a closed subgroup sheaf of
$\underline{\Aut}^G(U)$. Hence $\underline{\Aut}^G(X)$ is represented by a
closed subgroup scheme
$\Aut^G(X)\subset \Aut^G(U)$.
Since $\k$ has characteristic zero, every group scheme locally of finite
type over $\k$ is smooth. Thus $\Aut^G(X)$ is smooth and locally of finite
type. This proves~\ref{main th item i}.

We next construct the morphism to $\Aut(\overline C)$. By
Proposition~\ref{prop:useful-curve}, for every smooth $\k$-scheme $S$ and
every
$\varphi\in\Aut^G(X)(S)$,
the restriction $\varphi|_{U_S}$ induces, via the $(G,H)$-fibration
$q_S\colon U_S\to C_S$, an automorphism of $C_S$. Thus we obtain a
morphism of group sheaves
$\underline{\Aut}^G(X)\longrightarrow\underline{\Aut}(C)$.
Restriction identifies $\underline{\Aut}(C)$ with the subgroup sheaf of
$\underline{\Aut}(\overline C)$ consisting of automorphisms preserving
$\overline C\setminus C$. Composing with this inclusion therefore defines
a morphism of $\k$-group schemes
\[
\rho\colon \Aut^G(X)\longrightarrow \Aut(\overline C).
\]
Its kernel consists precisely of the $G$-equivariant automorphisms acting
trivially on $\k(X)^G$. Thus, writing
$\Aut_{\overline C}^G(X)=\Ker(\rho)$,
we obtain the exact sequence
\[
1\longrightarrow \Aut_{\overline C}^G(X)
\longrightarrow \Aut^G(X)
\xrightarrow{\rho}
\Aut(\overline C).
\]
This proves~\ref{main th item ii}.

It remains to describe $\Aut_{\overline C}^G(X)$. Under the closed immersion
\[
\Aut^G(X)\hookrightarrow \Aut^G(U)
\simeq \Aut^\Delta(V),
\]
the subgroup $\Aut_{\overline C}^G(X)$ is contained in
$\Aut_C^\Delta(V)$: indeed, an element acting trivially on
$\k(X)^G=\k(C)$ induces the identity on the smooth $\k$-curve $C$. Moreover,
$\Aut_{\overline C}^G(X)$ is a closed subgroup scheme of
$\Aut_C^\Delta(V)$.

By Proposition~\ref{prop: Aut^D_Y pour des torseurs}, there is an exact
sequence of abelian group schemes
\[
1\longrightarrow \Delta
\longrightarrow \Aut_C^\Delta(V)
\longrightarrow \Lambda_0
\longrightarrow 1,
\]
where $\Lambda_0\simeq\Z^n$ is a constant \'etale $\k$-group scheme for
some $n\geq 0$.
Set
\[
B=\Aut_{\overline C}^G(X)
\qquad\text{and}\qquad
D=B\cap \Delta.
\]
Since $D$ is a closed subgroup scheme of the group of multiplicative type
$\Delta$, it is itself of multiplicative type.

Let $\Lambda$ be the image of $B$ in $\Lambda_0=\Z^n$. Then $\Lambda$
is the constant subgroup scheme corresponding to a subgroup of $\Z^n$.
Every subgroup of a finitely generated free abelian group is again
finitely generated and free. Hence
$\Lambda\simeq\Z^m$
for some $m\geq 0$. The kernel of the induced surjective morphism
$B\longrightarrow\Lambda$
is precisely $B\cap\Delta=D$. Hence
\[
1\longrightarrow D
\longrightarrow \Aut_{\overline C}^G(X)
\longrightarrow \Lambda
\longrightarrow 1
\]
is exact. This proves~\ref{main th item iii} over the algebraically closed
field $\k$.

Finally, all the constructions above are canonical and hence compatible
with the natural Galois action. They therefore descend to the original
field. Under this descent, $D$ remains a group of multiplicative type,
while the constant group scheme $\Lambda\simeq\Z^m$ becomes an \'etale
group scheme that is \'etale-locally isomorphic to $\Z^m$. This
completes the proof.
\end{proof}

\begin{remark} \label{rk: sphericality}
The sphericality assumption enters the proof in two places. First, it ensures
the existence of a stabilizer in general position $H\subset G$ and hence,
after restricting to a suitable dense open subset, of a $(G,H)$-fibration.
Second, it implies that
$\Aut^G(G/H)\simeq N_G(H)/H$
is of multiplicative type, which allows us to apply
Proposition~\ref{prop: Aut^T pour des torseurs}. Thus the argument extends more
generally to actions whose general orbits have codimension one, provided that
a stabilizer in general position exists and that the corresponding group
$N_G(H)/H$ satisfies an appropriate representability statement for torsors.
\end{remark}

\section{Finiteness of real forms}\label{sec: real forms}
In this section, we use our description of equivariant automorphism groups obtained in Theorem~\ref{thm:full-representability} to prove the finiteness of real forms of arbitrary complexity-one $G$-varieties.

\smallskip

Let $\k$ be a perfect field, let $G$ be an algebraic $\k$-group, and let $X$ be a $G$-variety. We say that a $G$-variety $X'$ is a \emph{$(\k,G)$-form} of $X$ if $X_{\bk}$ and $X'_{\bk}$ are isomorphic as $G_{\bk}$-varieties, where $\bk$ denotes an algebraic closure of $\k$. There is a natural bijection
\[
\left\{\text{$(\k,G)$-forms of $X$ up to $G$-isomorphism}\right\}
\simeq
\H^1(\k,\Aut^G(X));
\]
see \cite[III, \S1]{Ser02} and \cite[III, Th.~2.5.1]{Giraud}. See also \cite[\S8]{Wed18} and \cite[\S2.3]{MJT24} for a more explicit treatment in the case of $G$-varieties. Thus, to prove that $X$ admits only finitely many $(\k,G)$-forms, it suffices to prove that $\H^1(\k,\Aut^G(X))$ is finite.

\smallskip

For almost homogeneous varieties, we obtain a result over any field of type $(F)$, and not only over $\R$. Recall from \cite[III, \S4.2]{Ser02} that a field $\k$ is said to be \emph{of type $(F)$} if it is perfect and, for every integer $n\geq2$, there are only finitely many subextensions of $\bk/\k$ of degree $n$. For instance, real closed fields, $p$-adic fields, and finite fields are of type $(F)$.

\begin{corollary}[of Theorem~\ref{th: aut of an almost homogeneous variety}]
\label{cor: forms for almost homog varieties}
Let $\k$ be a field of type $(F)$, let $G$ be a linear algebraic $\k$-group, and let $X$ be an almost homogeneous $G$-variety. Then $X$ admits only finitely many $(\k,G)$-forms.
\end{corollary}

\begin{proof}
By Theorem~\ref{th: aut of an almost homogeneous variety}, the group $\Aut^G(X)$ is a linear $\k$-group. Since $\k$ is of type $(F)$, the finiteness of $\H^1(\k,\Aut^G(X))$ follows from \cite[III, \S4.3, Th.~4]{Ser02}.
\end{proof}

We now specialize to $\k=\R$ and write $\Gamma=\Gal(\C/\R)$.
In view of Theorem~\ref{thm:full-representability}, the equivariant
automorphism groups that occur below are smooth group schemes, locally of
finite type, with finitely generated virtually abelian groups of connected
components. 

The following proposition shows that, for an $\R$-group scheme
locally of finite type, the latter condition implies the finiteness of its
first Galois cohomology.

\begin{proposition}\label{prop: finiteness of H1R}
Let $\mathcal G$ be an $\R$-group scheme, locally of finite type. Assume that
\[
\pi_0(\mathcal G)(\C)
=
\mathcal G(\C)/\mathcal G^\circ(\C)
\]
is a finitely generated virtually abelian group. Then
$\H^1(\R,\mathcal G)$ is finite.
\end{proposition}

\begin{proof}
We first consider the case where $\mathcal G$ is connected. 
If $\mathcal G$ is linear, the
assertion follows from \cite[III, \S4.3, Th.~4]{Ser02}, since $\R$ is of
type $(F)$.

If $\mathcal G=A$ is an abelian variety, restriction to $\C$ shows that
$\H^1(\R,A)$ is annihilated by $2$. The exact sequence
\[
1\to A[2]\to A\xrightarrow{[2]}A\to1
\]
therefore gives a surjection
\[
\H^1(\R,A[2])\longrightarrow\H^1(\R,A).
\]
Since $A[2]$ is finite, $\H^1(\R,A)$ is finite.

Since $\mathcal G$ is connected, Chevalley's theorem
gives an exact sequence
\[
1\to L\to\mathcal G\to A\to1,
\]
where $L$ is linear and $A$ is an abelian variety. Every twist of $L$ is
again a linear group, so the assertion follows from the two
preceding cases and \cite[I, \S5.5, Cor.~3]{Ser02}.

We next consider an \'etale $\R$-group scheme $\Lambda$ such that
$\Lambda(\C)$ is finitely generated and virtually abelian. Set
$\Gamma=\Gal(\C/\R)$ and $\Pi=\Lambda(\C)$.
Since $\Pi$ is finitely generated and virtually abelian, it contains an
abelian subgroup $B$ of finite index. A finitely generated group has only
finitely many subgroups of any given finite index. Hence the intersection
of all subgroups of $\Pi$ having index $[\Pi:B]$ is a characteristic
subgroup $M\subseteq B$ of finite index in $\Pi$. In particular, $M$ is
finitely generated and abelian. Since $M$ is characteristic, it is stable
under the action of $\Gamma$.

The subgroup $M$ therefore corresponds to an \'etale normal subgroup scheme
$\mathcal M\subseteq\Lambda$, and we have an exact sequence
\[
1\to\mathcal M\to\Lambda\to F\to1,
\]
where $F$ is a finite \'etale $\R$-group scheme.

For every twist $\mathcal M'$ of $\mathcal M$ arising from this extension,
the group $\mathcal M'(\C)$ is still finitely generated abelian. Since
$\Gamma$ has order $2$, the abelian group
$\H^1(\R,\mathcal M')$ is annihilated by $2$. Moreover, it is a quotient of
a subgroup of the finitely generated abelian group $\mathcal M'(\C)$.
Thus $\H^1(\R,\mathcal M')$ is a finitely generated $2$-torsion group, and
hence is finite. Since $\H^1(\R,F)$ is finite, it follows from
\cite[I, \S5.5, Cor.~3]{Ser02} that $\H^1(\R,\Lambda)$ is finite.

Finally, consider the exact sequence
\[
1\to\mathcal G^\circ\to\mathcal G\to\pi_0(\mathcal G)\to1.
\]
The group scheme $\mathcal G^\circ$ is connected, while
$\pi_0(\mathcal G)$ is \'etale and
$\pi_0(\mathcal G)(\C)$ is finitely generated and virtually abelian by
hypothesis. Every twist of $\mathcal G^\circ$ is again connected. 
The result therefore follows once more from
\cite[I, \S5.5, Cor.~3]{Ser02}.
\end{proof}

We can now prove the finiteness result announced in the introduction.

\begin{theorem}\label{thm: main result on real forms}
Let $G$ be a connected reductive $\R$-group, and let $X$ be a $G$-variety
of complexity one. Then $X$ admits only finitely many $(\R,G)$-forms.
\end{theorem}

\begin{proof}
If $X$ is almost homogeneous, the result follows from
Corollary~\ref{cor: forms for almost homog varieties}. We may therefore
assume that $X$ is not almost homogeneous.

By Theorem~\ref{thm:full-representability}, the sheaf
$\underline{\Aut}^G(X)$ is represented by a smooth $\R$-group scheme
$\Aut^G(X)$, locally of finite type, whose group of connected components is
finitely generated and virtually abelian. Hence
$\H^1(\R,\Aut^G(X))$
is finite by Proposition~\ref{prop: finiteness of H1R}. Since the
$(\R,G)$-forms of $X$ are classified by
$\H^1(\R,\Aut^G(X))$, the result follows.
\end{proof}

The preceding theorem only gives finiteness. We finish this article by showing,
in a simple example, how the exact sequences of
Theorem~\ref{thm:full-representability} can be used to compute the real forms
explicitly.

\begin{example}\label{ex: real forms of Gm2}
Let $T=C=\G_{\mathrm m,\R}$ and $X=T\times C$, where $T$ acts on the first
factor. Then we will show that $X$ admits exactly four pairwise non-isomorphic $(\R,T)$-forms.

Since $X\to C$ is the trivial $T$-torsor, a $T$-equivariant automorphism
acting trivially on $C$ is given by $(t,c)\mapsto(tu(c),c)$ for some
$u\colon C\to T$. As $\R[c,c^{-1}]^\times=\R^\times c^\Z$, we obtain
$\Aut_C^T(X)\simeq\G_{\mathrm m,\R}\times\Z$. Thus, in the notation of
Theorem~\ref{thm:full-representability}\ref{main th item iii},
$D=\G_{\mathrm m,\R}$ and $\Lambda=\Z$.

Moreover,
$\Aut_\R(C)\simeq\G_{\mathrm m,\R}\rtimes(\Z/2\Z)$, where the nontrivial
element acts by inversion, and the exact sequence
\[
1\to\Aut_C^T(X)\to\Aut^T(X)\to\Aut_\R(C)\to1
\]
splits via $\varphi\mapsto((t,c)\mapsto(t,\varphi(c)))$. Hence
\[
\Aut^T(X)\simeq
(\G_{\mathrm m,\R}\times\Z)\rtimes
\bigl(\G_{\mathrm m,\R}\rtimes(\Z/2\Z)\bigr).
\]
On connected components, the nontrivial element of $\Z/2\Z$ acts on $\Z$
by $n\mapsto-n$. Therefore
$\pi_0(\Aut^T(X))\simeq\Z\rtimes(\Z/2\Z)$ is the infinite dihedral group,
which is virtually cyclic but not abelian.

We now compute $\H^1(\R,\Aut^T(X))$ using the preceding split exact sequence.
The set $\H^1(\R,\Aut_\R(C))$ classifies the real forms of the real curve
$C=\G_{\mathrm m,\R}$. There are exactly three such forms, represented by
\[
C_0=\Spec(\R[x,y]/(xy-1)),\qquad
C_1=\Spec(\R[x,y]/(x^2+y^2-1)),
\]
\[
C_2=\Spec(\R[x,y]/(x^2+y^2+1)).
\]
Here $C_1$ is the anisotropic one-dimensional torus $\Sone$, while $C_2$
is its unique nontrivial $\Sone$-torsor.

For the class corresponding to $C_i$, the standard fiber description for
non-abelian Galois cohomology identifies the fiber of
\[
\H^1(\R,\Aut^T(X))\longrightarrow \H^1(\R,\Aut_\R(C))
\]
with
\[
\H^1(\R,{}_{C_i}\Aut_C^T(X))/\Aut_\R(C_i),
\]
where ${}_{C_i}\Aut_C^T(X)$ denotes the corresponding twist. Since the
splitting above acts trivially on the $T$-factor, this twisted group is the
equivariant automorphism group of the trivial $T$-torsor over $C_i$. Its
group of complex points is $\mathcal O((C_i)_\C)^\times$ with the natural
Galois action. Since $(C_i)_\C\simeq\G_{m,\C}$ and hence
$\Pic((C_i)_\C)=0$, the Hochschild--Serre sequence yields
\[
\H^1(\R,{}_{C_i}\Aut_C^T(X))\simeq\Pic(C_i).
\]
This identification is equivariant for the natural action of
$\Aut_\R(C_i)$, and therefore the fiber over $C_i$ is
$\Pic(C_i)/\Aut_\R(C_i)$.

Now $\Pic(C_0)=0$, $\Pic(C_1)\simeq\Z/2\Z$, and $\Pic(C_2)=0$. The action
of $\Aut_\R(C_i)$ on these groups is trivial: this is immediate for $C_0$
and $C_2$, while $\Aut(\Z/2\Z)=1$. Thus the three fibers have respectively
$1$, $2$, and $1$ elements, and therefore
\[
\#\H^1(\R,\Aut^T(X))=4.
\]
Equivalently, the four forms are the trivial $T$-torsors over $C_0$, $C_1$,
and $C_2$, together with the unique nontrivial $T$-torsor over $C_1$.
\end{example}

\bibliographystyle{abbrv}
\bibliography{biblio}

\end{document}